\documentclass{article}[9pt]
\usepackage{theorem}
\usepackage{amssymb}
\usepackage{float}
\usepackage{amsmath}
\usepackage{amsfonts}
\usepackage[american]{babel}
\usepackage{verbatim}
\usepackage{geometry}
\usepackage{lipsum}
\usepackage{color}
\usepackage{hyperref}

\numberwithin{equation}{section}

\newcommand{\R}        {\mathbb {R}}

\newcommand{\N}        {\mathbb {N}}

\newcommand{\eps}      {\epsilon}
\newcommand{\lap}      {\bigtriangleup}

\newcommand{\noi}      {\noindent}

\newcommand{\x}        {\rho}

\newtheorem{theorem}{Theorem}[section]

\newtheorem{lemma}{Lemma}[section]
\newtheorem{proposition}[theorem]{Proposition}

\newtheorem{definition}{Definition}[section]
{\theorembodyfont{\rmfamily}
\newtheorem{remark}[theorem]{Remark}
}
\newenvironment{proof}{\noindent\textbf{Proof.}}{\hfill$\square$\medskip}

\newcommand*\rn{\mathbb{R}^n}

\begin{document}

\title{On the Asymptotic Analysis of Problems Involving Fractional Laplacian in Cylindrical Domains Tending to Infinity}
\maketitle

\centerline{\scshape  Indranil Chowdhury$^1$ and Prosenjit Roy$^1$ }
\medskip
{\footnotesize

  \centerline{1 Tata Institute of Fundamental Research, Centre For Applicable Mathematics, Bangalore-560065, India.}
   \centerline{indranil@math.tifrbng.res.in, prosenjit@math.tifrbng.res.in}

} 

\vspace*{.5cm}

\begin{abstract}The article is an attempt to investigate the issues of asymptotic analysis for problems involving fractional Laplacian where the domains tend to become unbounded in one-direction. Motivated from the pioneering work on  second order elliptic problems by Chipot and Rougirel in \cite{CR}, where the force functions are considered on the cross section of domains, we prove the non-local counterpart of their result. 

Furthermore, recently in  \cite{Karen} Yeressian established a weighted estimate 
for solutions of nonlocal Dirichlet problems which exhibit the asymptotic 
behavior. The case when $s=1/2$ was also treated as an example to 
show  how the weighted estimate might be used to achieve the asymptotic behavior. In this article, we extend this result to each order between $0$ and $1$.
\end{abstract}

\let\thefootnote\relax
\let\thefootnote\relax 

\section{Introduction} \label{Introduction}

The non-local operators, in particular the fractional Laplace operators, have gained a great interest in recent day research; both for their interesting theoretical structures and for their wide range of applications. Motivated from the recent interest in this topic, we consider the non-local counterpart of the following second order elliptic problem :
\begin{equation}\label{Aux Problem}
  \left\{
 \begin{array}{ll}
 -\lap  v_\ell =  f_\ell \hspace{8.7mm}&\textrm{in} \ \Omega_\ell , \\
  v_\ell = 0   \hspace{16.9mm} &\textrm{on}  \ \partial \Omega_\ell, 
\end{array}
\right.
\end{equation}
where, $\Omega_{\ell}= (-\ell,\ell)\times \omega \subset\rn$ denotes a cylinder of length $\ell$ with the open bounded set $\omega \subset \R^{n-1}$ as the cross section and $v_{\ell}\in H_0^1(\Omega_{\ell})$ denotes the unique weak solution of \eqref{Aux Problem} for $f_{\ell} \in L^2(\Omega_{\ell})$. Asymptotic behavior of problems of type \eqref{Aux Problem}, as the length $\ell$ tends to infinity, have been studied with full generality in last two decades. We refer \cite{b}, \cite{CR}, \cite{karen}  as some of the relevant references in this topic. Also a large amount of literature is available on the asymptotic analysis for different types of elliptic, parabolic and hyperbolic problems involving different boundary conditions, for instance see \cite{d}, \cite{pro}, \cite{pr},  \cite{CS1}, \cite{CS2}, \cite{Sen1}, \cite{Sen2}. Our aim of this paper is to investigate the asymptotic behavior of the solution to the Dirichlet problem for the fractional Laplace operator. To be precise, let us consider the following problem : 
\begin{equation}\label{Main Problem}
  \left\{
 \begin{array}{ll}
 (-\lap )^s u_\ell =  f_\ell \hspace{8.7mm}&\textrm{in} \ \Omega_\ell , \\
  u_\ell = 0   \hspace{16.9mm} &\textrm{on}  \ \Omega_\ell^c. 
\end{array}
\right.
\end{equation}
Where, $\Omega_{\ell}$ is same as above definition and for a fixed $s \in (0,1)$ the fractional Laplace operator is defined as 
\begin{align}\label{fractional laplacian}
(-\lap)^{s}u(x)= C_{n,s} \ P.V \int_{\rn} \frac{u(x)- u(x+y)}{|y|^{n+2s}} \ dy .
\end{align} 
Here, $P.V$ denotes that the above integral has to be defined in principal value sense. The constant $C_{n,s}$ , which depends both on $n$ and $s$, will be defined explicitly in next section. The other way of defining the fractional Laplacian is also possible which has also been discussed in the next section.\\

 To the best of our knowledge, only result available on asymptotic analysis for non-local elliptic problems is due to Yeressian (see, \cite{Karen}).  In \cite{Karen} the author has obtained weighted estimates for solutions of 
Dirichlet problems for a class of non-local operators which are the generators of the semi-group created by symmetric pure jump Levy processes. 
Fractional Laplacians are generators of such processes. 
But in \cite{Karen} the weighted estimates has been used only to show 
the asymptotic behavior for the case $s=1/2$ (see, Lemma 4 of \cite{Karen}) . We can formulate this asymptotic behavior result as follows :
\begin{theorem} \label{Karen}
Let $u_{\ell}$ be the weak solution of the equation \eqref{Main Problem} for $s = \frac{1}{2}$ with the condition that 
\begin{align}\label{force term condn}
support(f_{\ell}) \subset \Omega_{\ell} \setminus \Omega_{\ell -1} \quad and \quad ||f_{\ell}||_{L^2(\Omega_{\ell})} \leq K,
\end{align}
where $K>0$ is a constant independent of $\ell$. Then, as $\ell \rightarrow \infty$ we have 
\begin{align*}
\int_{\Omega_{1}} u_{\ell}^2 \longrightarrow \ 0. 
\end{align*} 
\end{theorem}

In this article, we deal with two kinds of problems. In the first part, we consider a completely different condition  than \eqref{force term condn} on the force term $f_{\ell}$. In this case, similar to the result by Chipot and Rougirel in \cite{CR}, the force term $f_{\ell}$ would be defined only on $\omega \subset \R^{n-1}$ which assures us that $f_{\ell}$ would be independent of $\ell$. To be precise, we denote $X \in \rn$ by $(x_1, X_2) \in \R \times \R^{n-1}$ and we will study the following problem 
\begin{equation}\label{Actual  Problem}
  \left\{
 \begin{array}{ll}
 (-\lap )^s u_\ell =  f(X_2) \hspace{8.7mm}&\textrm{in} \ \Omega_{\ell} , \\
  u_\ell = 0   \hspace{16.9mm} &\textrm{on}  \ \Omega_{\ell}^c, 
\end{array}
\right.
\end{equation}
 where, $f \in L^2(\omega).$\smallskip
 
 In the case of classical Laplacian, Chipot \textit{et al} in \cite{CR} and \cite{Karen} have showed that the weak solution $v_{\ell}$ of equation \eqref{Aux Problem} with $f_{\ell} = f(X_2)$ will converge to $v_{\infty}$ in $H^1(\Omega_{\alpha \ell})$ for any $\alpha \in (0,1)$ where $v_{\infty} \in H_0^1(\omega)$ is the unique weak solution of the following problem: 
\begin{align*}
- \lap v_{\infty}(X_2) &= f(X_2) \qquad \mbox{in} \quad \omega, \\
 v_{\infty}(X_2) & = 0 \qquad \qquad \mbox{on}  \quad \partial \omega. 
\end{align*} 
 We expect a similar convergence result for fractional order Laplacian case also, i.e, the asymptotic behavior of the solution $u_{\ell}$ to the problem \eqref{Actual Problem} should depend on the problem \eqref{Limiting  Problem} which sets on the cross section $\omega$ of the cylinder. In particular, we obtain the following theorem which is one of the main result in our article. Let $(-\lap')^s$ be denoted as fractional Laplace operator in $n-1$ dimension. 

\begin{theorem}\label{main th}
Let us assume $s \in (\frac{1}{2}, 1)$, $f(X_2) \in L^2(\omega)$. Let $u_{\ell}$  be the unique weak solution of equation \eqref{Actual Problem} for each $\ell$ and $u_\infty$ be the unique weak solution of the following equation on the cross section $\omega$ of the cylinder $\Omega_\ell$,
\begin{equation}\label{Limiting  Problem}
  \left\{
 \begin{array}{ll}
 (-\lap' )^s u_\infty =  f(X_2) \hspace{8.7mm}&\textrm{in} \ \omega , \\
  u_\infty = 0   \hspace{16.9mm} &\textrm{on}  \ \omega^c. 
\end{array}
\right.
\end{equation}
Then for each $\alpha \in(0,1)$ as  $\ell \rightarrow \infty$, we have $$\int_{\Omega_{\alpha\ell}} |u_\ell - u_\infty|^2 \rightarrow 0 . $$ 
\end{theorem}\smallskip

In the second part of this article, we reconsider the equations of type \eqref{Main Problem}  when $f_{\ell}$ satisfies \eqref{force term condn}. Our second main result of this article extends Theorem \ref{Karen} to every values of $s \in (0,1)$. In particular, we prove the following theorem :  
\begin{theorem}
\label{fd}
Assume $f_{\ell}$ satisfies the condition \eqref{force term condn} and let $u_{\ell}$ be the weak solution of the problem \eqref{Main Problem}. Then for any $\delta > 0$,  as $\ell \rightarrow \infty$ one has 
$$\int_{\Omega_1} u_\ell^2 \longrightarrow 0. $$

\end{theorem}

In the above theorem one can replace $\Omega_1$ with $\Omega_R$ for any given $R>0$. Here we would like to point out that the rate of convergence obtained while proving Theorem \ref{Karen} is better than the rate we obtained while proving Theorem \ref{fd} for the case of $s= \frac{1}{2}$. However, we are more interested in studying the convergence of $u_{\ell}$ in $L^2$-norm for each $s\in (0,1)$ rather than to get a better rate of convergence in particular. It makes Theorem \ref{fd} more valuable for our purpose. \\

One can deduce the convergence result for the solution of equation \eqref{Aux Problem} with $f_{\ell}= f(X_2)$ from the convergence result for the solution of equation \eqref{Aux Problem} with the condition \eqref{force term condn}. This can be done using a suitable transformation of the type \eqref{transformation}. Thus, in the case of classical Laplacian, the two problems are some what related. But, for the problem involving fractional Laplacian, the situation is completely different due to its non-local nature. Still, we can prove a connection between these two problems by considering the force function $f$ with better regularity. We will discuss this issue in details in the Appendix.\\

Rest of the article is organized as follows. In the next section we introduce some basic notations and definitions which would be required throughout the article. Section \ref{pre} is  devoted to some preliminary results that will be  required in the proof of Theorem \ref{main th}. We need to emphasize that $v_{\infty}$ can be extended to $\Omega_{\ell}$ for each $\ell >0$ trivially in the classical Laplace case and we can easily see that $\lap v_{\infty}$ is well defined in $\Omega_{\ell} \subset \rn$. However, in the non-local setting one has to justify the well-definedness properly and this issue has also been considered in Section \ref{pre}.  
In Section \ref{mao} we present the proof of Theorem \ref{main th}. In Section \ref{ss} we analyze the case of article \cite{Karen} where the force function $f_{\ell}$ has chosen to be supported inside $\Omega_{\ell} \setminus \Omega_{\ell -1}$ and we extend Theorem \ref{Karen} to all $s \in (0,1)$. The paper concludes with an Appendix where we make a connection between these two problems with some extra assumptions. 
 
\vspace{0.5cm}

\section{Notations  and Definitions}
In this section we introduce various notations and definitions of spaces that would be used through out the paper.
\smallskip

Unless and otherwise mentioned, $X,Y,Z$ will denote points in $\rn$, $x_1,y_1,z_1$ will denote points in $\R$ and $X_2,Y_2,Z_2$ will denote points on $\R^{n-1}$. It is understood that, through out this paper functions are  extended by zero, if not explicitly mentioned. $B_R$ will denote ball of radius $R$ with center at $0$. We use letters $C,K$ etc to denote various generic constants which may change from line to line.   \smallskip
 
 The $L^{\infty}$-norm in a set $U\subset \rn$ will be denoted by $||.||_{L^{\infty}(U)}$. Similarly the $L^2$-norm in $U$ will be denoted by $||.||_{L^2(U)}$. The H\"older space $C^{k,\alpha}(U)$ with $k \in \N$, $\alpha \in (0,1]$ is defined as the subspace of $C^k(U)$ consisting of functions whose $k$ - th order partial derivatives are uniformly H\"older continuous with exponent $\alpha$ and the norm will be defined as 
$$||u||_{C^{k,\alpha}(U)} = \sum _{l=0}^k ||D^l u||_{L^{\infty}(U)} + \sup_{X,Y \in U} \frac{|D^k u(X) - D^ku(Y)|}{|X-Y|^{\alpha}}. $$ 

 For $s\in (0,1)$, the $n$ and $n-1$ dimensional  fractional Laplacian operators will be  denoted by $(-\lap)^s$ and $(-\lap')^s$ respectively.  We have already defined  the fractional Laplacian in \eqref{fractional laplacian}. We can also write the integral in \eqref{fractional laplacian} as weighted second order differential quotient, provided it is well defined. For any $s\in (0,1)$ and $X \in \rn$ we write
 \begin{align}\label{fractional laplacian 1}
 (- \lap)^s u(X) = C_{n,s} \int_{\rn} \frac{2u(X)- u (X+Y) -u (X-Y)}{|Y|^{n+2s}} \ dY.
 \end{align}
 The novelty of this representation is that the above integral does not involve the singularity at origin. We refer to \cite{NGV} for the equivalence of two definitions \eqref{fractional laplacian} and \eqref{fractional laplacian 1}. The space $H^s(\R^n)$ is defined as the space of all functions $u \in L^2(\R^n)$, such that the map 
$$(X,Y) \mapsto  \frac{u(X) -u(Y)}{|X-Y|^{\frac{n}{2}+s}}$$
belongs to $L^2(\R^n\times \R^n)$. It is well known (see, \cite{kassman}) that $H^s(\R^n)$ is a Hilbert space endowed with the norm
\begin{equation} \label{norm}
||u||_{H^s(\R^n)} = \int_{\R^n}|u(X)|^2dX + \int_{\R^n\times \R^n}  \frac{|u(X) -u(Y)|^2}{|X-Y|^{n+2s}}dXdY.
\end{equation}
The second integral of \eqref{norm} is called Gagliardo semi-norm for $H^s(\R^n)$ and is denoted by  $[.]_{H^s(\R^n)}$.  Let $\Omega$ be any open bounded subset of $\R^n$. We define the space $H_\Omega^s(\R^n)$, endowed with the norm $||.||_{H^s(\R^n)}$, as
 
$$H_\Omega^s(\R^n) := \left\{ u \in H^s(\R^n) \ \big|  \ u = 0  \ \textrm{on} \ \Omega^c \right\}.$$
Also by $V(\Omega)$, we denote the space of all functions from $\R^n$ to $\R$, such that  $u\big|_\Omega \in L^2(\Omega)$ and the map 
$$(X,Y) \mapsto  \frac{u(X) -u(Y)}{|X-Y|^{\frac{n}{2}+s}}$$
belongs to $L^2(\Omega\times \R^n)$.  $C_c^\infty(\Omega)$ will denote the set of compactly supported smooth functions in $\Omega$. \smallskip

We define $\varepsilon_{\Omega} : H_\Omega^s(\R^n) \times H_\Omega^s(\R^n) \rightarrow \R$, the bi-linear form, as 
\begin{multline*}
\varepsilon_\Omega(u,v) := C_{n,s}
\int_{\R^n}\int_{\R^n} \frac{\left\{u(X)-u(Y)\right\} v(X)}{|X-Y|^{n+2s}}dXdY  \\= \frac{C_{n,s}}{2} \int_{\R^n}\int_{\R^n} 
\frac{\left\{u(X)-u(Y)\right\}\{ v(X) - v(Y)\}}{|X-Y|^{n+2s}}dXdY,
\end{multline*}
where $C_{n,s }$ can be written down explicitly as 
\begin{align}\label{constant}
C{n,s} = \frac{s2^s \Gamma\left(\frac{n+2s}{2} \right)}{\pi^{\frac{n}{2}} \Gamma(1-s)}
\end{align} and $\Gamma$ denotes the usual Gamma function. Detailed discussion on the constant $C_{n,s}$ can be found in \cite{NGV}.
We refer to \cite{kassman} for the the proof of the second equality in the definition of $\varepsilon_\Omega$ above. We now define the notion of weak solution to the problem \eqref{Main Problem}.

\begin{definition}
Let $\Omega_\ell$ be the open bounded set defined as above and $f_{\ell} \in L^2(\Omega_{\ell})$. Then, $u_{\ell} \in H_{\Omega_{\ell}}^s(\R^n)$ is a weak solution of the problem \eqref{Main Problem}, if it satisfies 
\begin{equation}
\label{main1}
\varepsilon_{\Omega_\ell}(u_\ell, \phi) = \int_{\Omega_\ell} f_{\ell} \phi,   \hspace{4mm} \forall \phi \in H_{\Omega_\ell}^s(\R^n). 
\end{equation}
\end{definition}

The weak solution of the problem \eqref{Limiting  Problem}  can be understood in similar way. Next we define the notion of weak solution for non-homogeneous boundary value problem of fractional Laplacian. 

\begin{definition}
Let $g \in V(\Omega)$ and $f' \in L^2(\Omega)$,  a function $u \in  V(\Omega)$ is called  a weak solution of 
\begin{equation*}\label{Main Problem 1}
  \left\{
 \begin{array}{ll}
 (-\lap )^s u =  f' \hspace{8.7mm}&\textrm{in} \ \Omega, \\
  u = g   \hspace{16.9mm} &\textrm{on}  \ \Omega^c, 
\end{array}
\right.
\end{equation*}
if $u-g \in H_\Omega^s(\R^n)$ and  
$$ \varepsilon_{\Omega}(u,\phi) = \int_{\Omega} f'\phi,   \hspace{4mm} \forall \phi \in H_{\Omega}^s(\R^n).$$
\end{definition}
 
Existence and uniqueness of the weak solutions for those problems are well studied and fairly well-known now a days. We refer to \cite{kassman} for further reference.

\section{Some Preliminary Results}\label{pre}

Before we start proving the preliminary results, required for the proof of Theorem \ref{main th}, one important fact that should be noticed is the involvement of the constant $C_{n,s}$(depends on
the dimension) in the weak formulation of the problem \eqref{Actual Problem} and which is not the case for Laplacian. Since the problem \eqref{Actual Problem} and \eqref{Limiting  Problem} deals
with fractional Laplacian in $n$ and $n-1$ dimensions respectively, necessarily there should be a connecting formula for the constants. Our next lemma provides the aforesaid connection.
\begin{lemma}\label{connection}
For each $n \in \N$ and $s\in (0,1)$, let $C_{n,s}$ be defined by \eqref{constant}. Then one has $C_{n,s} \Theta_n = C_{n-1,s}$, where 
$$\Theta_n =  \int_{\R} \frac{dz}{(1+z^2)^{\frac{n+2s}{2}}}.$$
\end{lemma}
\begin{proof}
First note that $ \Theta_n = 2\int_0^\infty \frac{dz}{(1+ z^2)^{\frac{n+2s}{2}}}$. Changing the variable by $z= \tan \theta$, one obtains 
$$\Theta_n = 2\int_0^{\frac{\pi}{2}}(\cos \theta )^{n+2s-2}d\theta = B\left(\frac{1}{2},\frac{n+2s-1}{2}\right),$$
where,  $$B(x,y) = 2\int_0^{\frac{\pi}{2}}(\cos \theta )^{2x-1} (\sin \theta )^{2y-1} d\theta,$$
for $x, y \in (0, \infty)$. Using the formula $ B\left(\frac{1}{2},\frac{n+2s-1}{2}\right) = \frac{\Gamma(\frac{1}{2})\Gamma(\frac{n+2s-1}{2}) }{\Gamma(\frac{n+2s}{2})}$ one obtains the desired result.
\end{proof}

As we have discussed in the introduction, for the problem involving Laplace operator it is a straight forward calculation to claim that $v_{\infty}$ is the unique weak solution of the following problem : 
\begin{equation*}
  \left\{
 \begin{array}{ll}
 -\lap  v =  f(X_2) \hspace{8.7mm}&\textrm{in} \ \Omega_\ell, \\
  v= v_\infty(X_2)   \hspace{16.9mm} &\textrm{on}  \ \partial \Omega_\ell. 
\end{array}
\right.
\end{equation*}
But in the case of fractional Laplacian we have to establish that $u_{\infty}$ is a weak solution of \eqref{Main Problem 2}. 

\begin{lemma} \label{limiting lemma}
Let us define $u_\infty^*(x_1, X_2) := u_\infty(X_2)$. Then $u_\infty^*$  is the unique weak solution of the problem 
\begin{equation}\label{Main Problem 2}
  \left\{
 \begin{array}{ll}
 (-\lap )^s u =  f(X_2) \hspace{8.7mm}&\textrm{in} \ \Omega_\ell, \\
  u = u_\infty(X_2)   \hspace{16.9mm} &\textrm{on}  \ \Omega_\ell^c, 
\end{array}
\right.
\end{equation}
 where $u_\infty$ is the weak solution of  \eqref{Limiting  Problem}.
\end{lemma}
\begin{proof}
First of all we would like to show that problem \eqref{Main Problem 2} is well defined, that is $u_\infty^* \in V(\Omega_\ell)$. Clearly $u_\infty^*\big|_{\Omega_\ell} \in L^2(\Omega_\ell)$.  Next we would like to show that 
$$\int_{\Omega_\ell}\int_{\R^n} \frac{|u_\infty^*(X)-u_\infty^*(Y)|^2}{|X-Y|^{n+2s}}dXdY < \infty.$$
We see, 
\begin{eqnarray*}
\int_{\Omega_\ell}\int_{\R^n} \frac{|u_\infty^*(X)-u_\infty^*(Y)|^2}{|X-Y|^{n+2s}}dXdY \leq 
\int_{-\ell}^\ell\int_\omega\int_{\R^n} \frac{|u_\infty(X_2)-u_\infty(Y_2)|^2}{|X_2-Y_2|^{n+2s}\left( 1 + \frac{|x_1-y_1|^2}{|X_2-Y_2|^2}\right)^{\frac{n+2s}{2}}}dYdX_2dx_1\\
\leq
 \int_{-\ell}^\ell\int_\omega\int_{\R^{n-1}}  \frac{|u_\infty(X_2)-u_\infty(Y_2)|^2}{|X_2-Y_2|^{n+2s}} 
 \left( \int_{\R} \frac{dy_1}{ \left(1 + \frac{|x_1-y_1|^2}{|X_2-Y_2|^2}\right)^{\frac{n+2s}{2}}} \right)dY_2dX_2dx_1 \\
 \leq
\Theta_n  \int_{-\ell}^\ell \left( \int_{\R^{n-1} \times \R^{n-1}}  \frac{|u_\infty(X_2)-u_\infty(Y_2)|^2}{|X_2-Y_2|^{n-1+2s}} dX_2dY_2 \right) dx_1 = 2\ell \Theta_n [u_\infty]_{H_\omega^s(\R^{n-1})},
\end{eqnarray*}
where $\Theta_n$ is as in Lemma \ref{connection}.
This ensures that $u_\infty^* \in V(\Omega_\ell)$.\smallskip

Now we show that $u_\infty^*$ is indeed a weak solution of the problem \eqref{Main Problem 2}.
Let $ \phi \in H_{\Omega_{\ell}}^s(\R^n)$
\begin{eqnarray}
\varepsilon(u_\infty^*,\phi) = C_{n,s}\int_{\R^n}\int_{\R^n} \frac{\{u_\infty^*(X)-u_\infty^*(Y)\}\phi(X)}{|X-Y|^{n+2s}}dXdY\nonumber\\
  = C_{n,s}\int_{\R^n}\int_{\R^n} \frac{  \{u_\infty(X_2)-u_\infty(Y_2)\}\phi(X)}{|X_2-Y_2|^{n+2s}\left( 1 + \frac{|x_1-y_1|^2}{|X_2-Y_2|^2}\right)^{\frac{n+2s}{2}}}dYdX.\nonumber
\end{eqnarray}
Integrating first with respect to $y_1$ and using the fact that $u_\infty$ satisfies \eqref{Limiting  Problem} weakly, one obtains
\begin{eqnarray*}
\varepsilon(u_\infty^*,\phi)  = C_{n,s}\Theta_n \int_{\R} \left( \int_{\R^{n-1}}\int_{\R^{n-1}} \frac{  \{u_\infty(X_2)-u_\infty(Y_2)\}\phi(x_1,X_2)}{|X_2-Y_2|^{n-1+2s}}dX_2dY_2\right)dx_1 \nonumber\\
= \frac{C_{n,s}\Theta_n}{C_{n-1,s}} \int_{\R} \left( \int_\omega f(X_2)\phi(x_1,X_2) dX_2 \right) dx_1 = \int_{\Omega_\ell} f(X_2)\phi(X)dX.
\end{eqnarray*}
To obtain the last equality we have used Lemma \ref{connection} and since the problem \eqref{Main Problem 2} admits unique solution, this completes the proof of the lemma. 
\end{proof}\\

The next theorem is the Poincar\'e inequality in fractional sobolev space for the domain of the form $\Omega_{\ell}$. We refer to Lemma 1 of \cite{Karen} for the proof. 

\begin{theorem}\label{poicare}
\emph[Poincar\'e Inequality] Let $\ell >0$ and $\Omega_{\ell}$ be the domain defined above. Then for every $u \in H_{\Omega_{\ell}}^s(\rn)$ there exists a constant $C>0$, independent of $\ell$, such that 
$$\int_{\Omega_\ell} u^2 \ dX \leq C \int_{\R^n}\int_{\R^n} \frac{\{u(X)-u(Y)\}^2}{|X-Y|^{n+2s}} \ dX \ dY. $$
\end{theorem}

We conclude this section by proving an a priori estimate using Poincar\'e inequality, that will be useful in proving Theorem \ref{main th}. 

\begin{lemma}\label{L2}
Let $f \in L^2(\omega)$ and $u_\ell \in H_{\Omega_{\ell}}^s(\rn)$ be the weak solution of the equation \eqref{Actual Problem}. Then for every $\ell$ there exists some constant $C> 0$, independent of $\ell$, such that 
$$\int_{\Omega_\ell} u_\ell^2 \leq  C \ ||f||_{L^2(\omega)}^2 \ell.$$
\end{lemma}

\begin{proof}
Using $\phi = u_\ell$ in \eqref{main1}, Poincar\'e inequality and H\"older's inequality, it holds for some constant $C >0 $ (independent of $\ell$)
$$ \int_{\Omega_\ell} u_\ell^2 \leq C\int_{\Omega_\ell} f(X_2)u_\ell \leq  C \sqrt{\ell}  ||f||_{L^2(\omega)} ||u_\ell ||_{L^2(\Omega_\ell)}.$$
This proves the lemma.
\end{proof}

\section{Proof of Theorem \ref{main th}}\label{mao}
In this section we prove Theorem \ref{main th} which is one of the main result of this article. The subsequent function $\rho_{\ell}$ would play an integral role to prove the theorem.  
Let $\alpha \in (0,1)$ and $\x_\ell : \R \rightarrow \R$ be a Lipschitz continuous function on $\R$, defined by
\begin{equation*}\label{Main Problem 3}
 \rho_\ell = \left\{
 \begin{array}{ll}
  1 \hspace{8.7mm}&\textrm{in} \  (-\alpha \ell, \alpha \ell) \\
  0    \hspace{8.7mm} &\textrm{on}  \  (-\ell, \ell)^c, 
\end{array}
\right.
\end{equation*} 
such that $0 \leq \rho_\ell \leq 1$ and
$|\rho_\ell^{'}| \leq \frac{C}{\ell}$ for some constant $C>0$ depending on $\alpha$. Before going to the proof of the theorem, let us prove an important lemma which would be required for the proof.

\begin{lemma}
\label{inspace} 
Let $u_{\infty}$ be the weak solution of \eqref{Limiting  Problem} and $\x_\ell$ be  defined as above, then the function $\phi_\ell(X) := u_\infty(X_2) \x^2_\ell(x_1) \in H_{\Omega_\ell}^s(\R^n)$.
\end{lemma}

\begin{proof}
 It is clear that $\phi_\ell \in L^2(\R^n)$, as we see
\begin{equation*}
\int_{\R^n} \phi_\ell^2(X)\ dX = \left( \int_{-\ell}^\ell \rho_\ell^4(x_1)dx_1\right) \left( \int_{\omega} u_\infty^2(X_2)dX_2\right) <\infty.
\end{equation*} 
Now let us consider the semi-norm of $H_{\Omega_\ell}^s(\R^n)$.
\begin{eqnarray}
[\phi_\ell]_{H_{\Omega_\ell}^s(\R^n)} &=& \int_{\R^n}\int_{\R^n} \frac{|\phi_\ell(X)- \phi_\ell(Y)|^2}{|X-Y|^{n+2s}}dXdY \nonumber\\
&=&  \int_{\Omega_\ell}\int_{\R^n} \frac{|\phi_\ell(X)- \phi_\ell(Y)|^2}{|X-Y|^{n+2s}}dXdY  +  \int_{\Omega_\ell^c}\int_{\R^n} \frac{|\phi_\ell(X)- \phi_\ell(Y)|^2}{|X-Y|^{n+2s}}dXdY \nonumber\\
&\leq& 2 \int_{\Omega_\ell}\int_{\R^n} \frac{|\phi_\ell(X)- \phi_\ell(Y)|^2}{|X-Y|^{n+2s}}dXdY + 
 \int_{\Omega_\ell^c}\int_{\Omega_\ell^c} \frac{|\phi_\ell(X)- \phi_\ell(Y)|^2}{|X-Y|^{n+2s}}dXdY \nonumber\\
 &=&  2I_1 +I_2. \nonumber
\end{eqnarray}
Since $\phi_\ell = 0$ outside $\Omega_\ell$ this implies that $I_2 = 0$ and therefore it is sufficient to show that $I_1$ is  a finite term. 
\begin{eqnarray}
I_1 &=&   \int_{\Omega_\ell}\int_{\R^n} \frac{|\phi_\ell(X)- \phi_\ell(Y)|^2}{|X-Y|^{n+2s}}dXdY \nonumber \\
&\leq & 2 \int_{\Omega_\ell}\int_{\R^n} \frac{u_\infty^2(X_2)\{\x^2_\ell(x_1) -\x^2_\ell(y_1)\}^2}{|X-Y|^{n+2s}}dXdY
+  2\int_{\Omega_\ell}\int_{\R^n} \frac{\x_\ell^4(y_1)\{u_\infty(X_2) - u_\infty(Y_2)\}^2}{|X-Y|^{n+2s}}dXdY\nonumber\\
&=&2I_3 +2I_4\nonumber.
\end{eqnarray}
We first estimate the term $I_3$.
\begin{align*}
I_3 &= \int_{\Omega_\ell}\int_{|X-Y|< 1} \frac{u_\infty^2(X_2)\{\x^2_\ell(x_1) -\x^2_\ell(y_1)\}^2}{|X-Y|^{n+2s}}dXdY \nonumber\\
& \hspace{4cm}+ \int_{\Omega_\ell}\int_{|X-Y|\geq 1} \frac{u_\infty^2(X_2)\{\x^2_\ell(x_1) -\x^2_\ell(y_1)\}^2}{|X-Y|^{n+2s}}dXdY\nonumber\\
&\leq  C \int_{\Omega_\ell}\int_{|X-Y|< 1} \frac{u_\infty^2(X_2)|X-Y|^2}{|X-Y|^{n+2s}}dXdY +C  \int_{\Omega_\ell}\int_{|X-Y|\geq 1} \frac{u_\infty^2(X_2)}{|X-Y|^{n+2s}}dXdY\nonumber\\
&\leq   C\int_{\Omega_\ell} u_\infty^2(X_2) \ dX \int_{B(0,1)}\frac{dZ}{|Z|^{n+2s-2}} + C \int_{\Omega_\ell} u_\infty^2(X_2) \ dX \int_{B(0,1)^c}\frac{dZ}{|Z|^{n+2s}} \nonumber\\
&= K.
\end{align*}
The term $I_4 $  is  also finite  after following a similar argument as done in the beginning of the proof of  Lemma \ref{limiting lemma}. It concludes the proof.
\end{proof}\bigskip

\textbf{Proof of Theorem \ref{main th}.} \  Set $v_\ell := u_\ell -u_\infty$. Using the definition of $u_{\ell}$ and Lemma \ref{limiting lemma} we see that $v_\ell$ satisfies the following equation in weak sense:

\begin{equation*}\label{Main Problem 4}
  \left\{
 \begin{array}{ll}
 (-\lap )^s v_\ell
 =  0 \hspace{8.7mm}&\textrm{in} \ \Omega_\ell, \\
  v_\ell = - u_\infty(X_2)   \hspace{16.9mm} &\textrm{on}  \ \Omega_\ell^c, 
\end{array}
\right.
\end{equation*}
which is obtained after subtracting \eqref{Actual  Problem} from \eqref{Main Problem 2}.   Hence for $\phi\in H_{\Omega_\ell}^s(\R^n)$, one has 

\begin{equation}
\label{weak}
\frac{1}{2}C_{n,s}\int_{\R^n}\int_{\R^n} \frac{\{v_\ell(X)-v_\ell(Y)\}\{\phi(X) -\phi(Y)\}}{|X-Y|^{n+2s}}dXdY = 0.
\end{equation}

From previous lemma we know that $\phi_\ell := \rho_\ell^2(x_1)v_\ell(X) \in H_{\Omega_\ell}^s(\R^n) $, and hence the expression \eqref{weak} holds true with $\phi = \phi_\ell $. This implies that 
\begin{eqnarray}
0 &=& \int_{\R^n}\int_{\R^n} \frac{\{ v_\ell(X)-v_\ell(Y)\}\{ v_\ell(X)\x_\ell^2(x_1) -v_\ell(Y)\x_\ell^2(y_1) \}}{|X-Y|^{n+2s}}\ dX\ dY  \nonumber \\
&=&  \int_{\R^n}\int_{\R^n} \left[ \frac{\{v_\ell(X)-v_\ell(Y)\}^2 \rho_\ell^2(x_1)}{|X-Y|^{n+2s}}  +  \frac{  v_\ell(Y) \{ v_\ell (X) -v_\ell(Y) \}   \{ \x_\ell^2(x_1)- \rho_\ell^2(y_1)\} }{|X-Y|^{n+2s}} \right] dX\ dY.\nonumber
\end{eqnarray}
Therefore,
\begin{align*}\label{eq}
&\int_{\R^n}\int_{\R^n} \frac{\{v_\ell(X)-v_\ell(Y)\}^2 \rho_\ell^2(x_1)}{|X-Y|^{n+2s}} dX\ dY \\
&= - \int_{\R^n}\int_{\R^n} \frac{  v_\ell(Y) \{ v_\ell (X) -v_\ell(Y) \}   \{ \x_\ell^2(x_1)- \rho_\ell^2(y_1)\} }{|X-Y|^{n+2s}} dX dY\\
&\leq  \eps \int_{\R^n}\int_{\R^n} \frac{  \{v_\ell(X) -v_\ell(Y) \}^2\{\x_\ell(x_1) + \x_\ell(y_1)\}^2 }{|X-Y|^{n+2s}} dX dY+\frac{1}{\eps} \int_{\R^n}\int_{\R^n}\frac{\{ \x_\ell(x_1)_-\x_\ell(y_1)\}^2v_\ell^2(Y) }{|X-Y|^{n+2s}} dX dY\\ 
\end{align*}
\begin{align*}
&\leq  2\eps  \int_{\R^n}\int_{\R^n} \frac{ \{v_\ell(X) -v_\ell(Y) \}^2 \{ \x_\ell^2(x_1) + \x_\ell^2(y_1)\}}{|X-Y|^{n+2s}} dX dY+ \frac{1}{\eps} \int_{\R^n}\int_{\R^n}\frac{\{ \x_\ell(x_1)_-\x_\ell(y_1)\}^2v_\ell^2(Y) }{|X-Y|^{n+2s}} dX dY\\
&= 4\eps  \int_{\R^n}\int_{\R^n} \frac{ \{v_\ell(X) -v_\ell(Y) \}^2 \x_\ell^2(x_1) }{|X-Y|^{n+2s}} dX\ dY + \frac{1}{\eps} \int_{\R^n}\int_{\R^n}\frac{\{ \x_\ell(x_1)_-\x_\ell(y_1)\}^2v_\ell^2(Y) }{|X-Y|^{n+2s}} dX\ dY.
\end{align*}
Choosing $\eps$ small enough, one gets for some constant $C> 0$,
\begin{equation}
\label{eq2}
\int_{\R^n}\int_{\R^n} \frac{\{v_\ell(X)-v_\ell(Y)\}^2 \rho_\ell^2(x_1)}{|X-Y|^{n+2s}} dX\ dY\leq C
\int_{\R^n}\int_{\R^n}\frac{\{ \x_\ell(x_1)_-\x_\ell(y_1)\}^2v_\ell^2(Y) }{|X-Y|^{n+2s}} dX\ dY.
\end{equation}
Now applying Poincar\'e inequality, obtained from Theorem \ref{poicare}, we get that for some constant $C > 0$ (independent of $\ell$)
\begin{align}\label{poi}
&\int_{\Omega_\ell} (v_\ell\x_\ell)^2 \ dX \leq  C \int_{\R^n}\int_{\R^n \setminus \{0\}} \frac{\{ v_\ell(X)\x_\ell(x_1)-v_\ell(Y)\x_\ell(y_1)\}^2}{|X-Y|^{n+2s}}dXdY \nonumber \\
\leq & 2C\int_{\R^n}\int_{\R^n}\frac{\{ \x_\ell(x_1)_-\x_\ell(y_1)\}^2v_\ell^2(Y) }{|X-Y|^{n+2s}} dX\ dY+
2C\int_{\R^n}\int_{\R^n}\frac{\{ v_\ell(X)_-v_\ell(Y)\}^2\x_\ell^2(x_1) }{|X-Y|^{n+2s}} dX\ dY .
\end{align} 
Combining this \eqref{poi} with \eqref{eq2} and using Tonelli's theorem, we obtain for some constant $C > 0 $,

\begin{eqnarray*}
\int_{\Omega_\ell} (v_\ell\x_\ell )^2 \ dX &\leq & C\int_{\R^n}\int_{\R^n}\frac{\{ \x_\ell(x_1)_-\x_\ell(y_1)\}^2v_\ell^2(Y) }{|X-Y|^{n+2s}} \ dX\ dY \nonumber\\  &=& C \int_\R\int_{\R^n} \{ \x_\ell(x_1)_-\x_\ell(y_1)\}^2v_\ell^2(Y) \left(\int_{\R^{n-1}}\frac{1}{|X-Y|^{n+2s}}dX_2  \right) dx_1 \ dY\nonumber \\
&=&\frac{C}{\prod_{i=2}^n\Theta_{i}} \int_{\R^n}v_\ell^2(Y)\left(  \int_\R \frac{ \{ \x_\ell(x_1)_-\x_\ell(y_1)\}^2}{|x_1-y_1|^{1+2s}} dx_1 \right) dY \nonumber \\
&=& \frac{C}{\prod_{i=2}^n\Theta_{i}} \int_{\R^n}v_\ell^2(Y) J_\ell(y_1)dY,
\end{eqnarray*}
where  $$J_\ell(y_1) =  \int_\R \frac{ \{ \x_\ell(x_1)_-\x_\ell(y_1)\}^2}{|x_1-y_1|^{1+2s}} dx_1.$$
By definition of $v_{\ell}$, one obtains for some constant $C > 0$,
\begin{equation}
\label{fur}
\int_{\Omega_\ell} |v_\ell\x_\ell |^2  \ dY \leq  2C \int_{\R^n} u_\ell^2(Y)J_\ell(y_1) \ dY +  2C \int_{\R^n} u_\infty^2(Y_2) J_\ell(y_1) \ dY := 2C I_{\ell} +2C I_\infty.
\end{equation}
We now estimate the term $J_\ell(y_1)$. \\

\noi\underline{ \textbf{ Case 1}:  \hspace{4mm} $ y_1 \in (-2\ell, 2\ell)^c $} \\

In this case $\rho_\ell(y_1) = 0$, therefore
\begin{equation}
|J_{\ell}(y_1)  | \leq   \  ||\rho_\ell ||_{L^{\infty}}^2   \int_{-\ell}^{\ell} \frac{  dx_1}{|x_1-y_1|^{1+2s}} \leq \frac{C}{|\ell - y_1|^{2s}} + \frac{C}{|\ell + y_1|^{2s}}.\nonumber
\end{equation}
\noi\underline{ \textbf{ Case 2}:  \hspace{4mm} $ y_1 \in (-2\ell, 2\ell) $} \\

We estimate  this case using mean value theorem and $|\rho_\ell^{'} | \leq \frac{C}{\ell}$. We rewrite the integral as 
\begin{align*}
J_{\ell}(y_1) = \int_{\R} \frac{\{ \x_{\ell}(y_1+z_1) -\x_{\ell}(y_1) \}^2}{|z_1|^{1+2s}} \ dz_1.
\end{align*} 
Hence, 
\begin{align}\label{x}
|J_{\ell}(y_1)| & \leq C \int_{-\ell}^\ell \frac{ \{ \x_\ell(y_1+z_1)_-\x_\ell(y_1)\}^2}{|z_1|^{1+2s}} \ dz_1 + C\int_{(-\ell,\ell)^c}\frac{ \{ \x_\ell(y_1+z_1)_-\x_\ell(y_1)\}^2}{|z_1|^{1+2s}} \ dz_1 \nonumber\\
& \leq \frac{C}{\ell^2}  \int_{-\ell}^\ell \frac{ |z_1|^{2}}{|z_1|^{1+2s}} \ dz_1+ 2C||\rho_\ell ||_\infty^2 \int_{(-\ell,\ell)^c}\frac{ dz_1 }{|z_1|^{1+2s}}  \quad \leq 
\frac{C}{\ell^{2s}}.\nonumber
\end{align}
Combining the two cases we get 

\begin{equation}\label{Main Problem f3}
 J_\ell(y_1) \leq \left\{
 \begin{array}{ll}
   \frac{C}{\ell^{2s}} \hspace{8.7mm}&\textrm{in} \  (-2\ell, 2\ell), \\
  \frac{C}{|\ell - y_1|^{2s}} + \frac{C}{|\ell + y_1|^{2s}} \hspace{16.9mm} &\textrm{on}  \  (-2\ell,2\ell)^c. 
\end{array}
\right.
\end{equation}
Now we estimate the term $I_\infty$ in \eqref{fur} using the above estimate \eqref{Main Problem f3} for $J_\ell$.
\begin{multline}\label{Main Problem I1}
I_\infty = \int_{\R^n} u_\infty^2(Y)J_\ell(y_1)\ dY = \int_\omega u_\infty^2(Y_2)dY_2  \int_{-\infty}^{\infty}J_\ell(y_1)dy_1 \\ 
= ||u_\infty||_{L^2(\omega)}^2 \left\{ \int_{-2\ell}^{2\ell} J_\ell(y_1) dy_1+ \int_{(-2\ell,2\ell)^c} J_\ell(y_1)dy_1\right\} \leq \frac{C||u_\infty||_{L^2(\omega)}^2 }{\ell^{2s-1}},
\end{multline}
if $s > \frac{1}{2}$. Finally we estimate the term $I_\ell$, using Lemma \ref{L2}. As $u_\ell =0$ outside $\Omega_\ell$, we have 
\begin{equation}\label{Main Problem I2}
I_\ell = \int_{\Omega_\ell}u_\ell^2 J_\ell  \leq   \frac{C}{\ell^{2s}} \int_{\Omega_\ell} u_\ell^2 \leq  \frac{C ||f||_{L^2(w)}}{\ell^{2s-1}}.
\end{equation}
Finally, since $\rho_\ell = 1$ on $(-\alpha \ell,\alpha \ell)$, we obtain from \eqref{fur} by using \eqref{Main Problem I1} and \eqref{Main Problem I2} that 
 
$$\int_{\Omega_{\alpha \ell}}(u_\ell -u_\infty)^2 \ dX \leq \frac{C}{\ell^{2s-1} },$$
where the constant $C$ is independent of $\ell$. Since $s> \frac{1}{2}$, this concludes the proof of the theorem by letting $\ell \rightarrow \infty$. \quad \hfill$\square$

 \section{Proof of Theorem \ref{fd}}\label{ss}

In this section, we are going to study the second part of this article, i,e, we generalize Theorem \ref{Karen} for every $s \in (0,1)$. For a bounded and Lipschitz continuous function $\rho$ on $\rn$, similar to the article \cite{Karen}, let us define $S_s(\rho)$  by  
\begin{equation*}\label{s}
 S_s\left( \rho\right)(X) := \int_{\R^n \setminus \{0\}}\left(\sqrt{\rho(X+Y)} - \sqrt{\rho(X)}\right)^2\frac{dY}{|Y|^{n+2s}}.
\end{equation*}
Our aim is to construct a function $\rho$ for each $s\in (0,1)$ which satisfies 
\begin{align} \label{rho inequality}
S_s(\rho)(X) \leq C_0 \gamma  \rho(X) , \ \forall X \in \rn \quad \mbox{where,} \quad 0<\gamma < \frac{1}{10}.
\end{align}
Once we construct suitable $\rho$ which satisfies \eqref{rho inequality}, then the proof of Theorem \ref{fd} would be an easy consequence of following proposition.
\begin{proposition} \label{KAREN1}
Let $u_{\ell}\in H^s_{\Omega_{\ell}}(\rn)$ be the weak solution of \eqref{Main Problem}. If for some $\gamma < \frac{1}{10}$ there exist a bounded Lipschitz continuous function $\rho$ such that \eqref{rho inequality} holds. Then for some constant $C^1_{\gamma}>0$ independent of $\ell$, we have 
\begin{align*}
\int_{\Omega_{\ell}} u_{\ell}^2(X) \rho(X) \ dX \leq C_{\gamma}^1 \int_{\rn} f_{\ell}^2(X) \rho(X) \ dX. 
\end{align*}
\end{proposition}
The proof of Proposition \ref{KAREN1} is a straightforward application of Lemma 3 and Theorem 1 of article \cite{Karen} by Yeressian. \\
 
To this end, for $0 < \eps \leq 2$  let us define 
$$\phi_\eps(t) = \min \left\{ \frac{1}{2} ,\  \frac{1}{|t|^\eps + 1} \right\}.$$ 
Clearly $\phi_\eps$ is Lipschitz continuous and bounded. Then the function $\rho_{\eps, \lambda}$ on $\R^n$ defined  by
\begin{equation*}\label{ssd}
\rho_{\eps, \lambda}(X) := \phi_\eps \left(\frac{x_1}{\lambda}\right)
\end{equation*}
is also Lipschitz continuous and bounded. The next theorem can be considered as the main step for the proof of Theorem \ref{fd}.

\begin{theorem}
 \label{S}
For each $\lambda >0$ and $\eps < 2s$ one has for some constant $C = C(\eps) > 0$,  $$S_s\left( \rho_{\eps,\lambda}\right)(X) \leq \frac{C_\eps}{\lambda^{2s}}\rho_{\eps,\lambda}(X).$$ 
\end{theorem}

\begin{remark}
For $\eps \in (1,2]$ we can  also take $\phi_\eps (t) := \frac{1}{1+ |t|^\eps}$, then also  the conclusion of Theorem \ref{S} remains true. One can adopt a similar line of proof for this case.  
\end{remark}

We need the following two lemmas in order to prove Theorem \ref{S}. First let us prove these lemmas.

\begin{lemma}
\label{sup}
For $\eps >0$, there exist some constant $K =K(\eps) > 0$ such that, $$\sup_{z \in \R, |\tau | \geq 1} \frac{\phi_\eps(z +\tau)}{\phi_\eps(z) |\tau|^\eps} = K < \infty.$$ 
\end{lemma}
\begin{proof}
We write 
\begin{align*}
\sup_{z \in \R, |\tau | \geq 1} \frac{\phi_\eps(z +\tau)}{\phi_\eps(z) |\tau|^\eps} \leq \max \{A, B, C \}
\end{align*} 
where, 
$$A = \sup_{|z| \leq 2, |\tau | \geq 1} \frac{\phi_\eps(z +\tau)}{\phi_\eps(z) |\tau|^\eps}, \  B = \sup_{ |z| \geq 2, \  1 \leq |\tau |\leq \frac{|z|}{2}}
\frac{\phi_\eps(z +\tau)}{\phi_\eps(z) |\tau|^\eps} \  \textrm{and} \ C =\sup_{|z| \geq 2, \  |\tau | \geq \frac{|z|}{2}} \frac{\phi_\eps(z +\tau)}{\phi_\eps(z) |\tau|^\eps}.$$
Using $\phi_\eps(z) \geq \phi_\eps(2)$ on the set $|z| \leq 2$ and $||\phi_\eps||_{L^{\infty}(\R)} \leq \frac{1}{2}$, we get trivially 
\begin{equation}\label{a}
A \leq \frac{1+2^\eps}{2}.
\end{equation}
Now we estimate the term $B$. On the set $|\tau | \leq \frac{|z|}{2}$, using triangle inequality we get $|z + \tau | \geq \frac{|z|}{2}$. Using this we have 
\begin{multline}\label{b}
B = \sup_{ |z| \geq 2, \  1 \leq |\tau |\leq \frac{|z|}{2}}
\frac{\phi_\eps(z +\tau)}{\phi_\eps(z) |\tau|^\eps} = \sup_{ |z| \geq 2, \  1 \leq |\tau |\leq \frac{|z|}{2}}\frac{1+ |z|^\eps}{(1+|z+\tau|^\eps)|\tau|^\eps}\\
\leq 
\sup_{ |z| \geq 2, \  1 \leq |\tau |\leq \frac{|z|}{2}}\frac{1+ |z|^\eps}{|z+\tau|^\eps} 
\leq  2^\eps \sup_{ |z| \geq 2, \  1 \leq |\tau |\leq \frac{|z|}{2}}\frac{1+ |z|^\eps}{|z|^\eps} \leq 2^\eps + 1.
\end{multline}
Now we estimate the term $C$, using $||\phi_{\varepsilon}||_{L^{\infty}(\R)}\leq \frac{1}{2}$  we see
\begin{equation}
\label{c}
C \leq 2^\eps ||\phi_{\varepsilon}||_{L^{\infty}(\R)}  \sup_{|z| \geq 2, \  |\tau | \geq \frac{|z|}{2}} 1 + \frac{1}{|z|^\eps} \leq \frac{2^\eps + 1}{2}.
\end{equation}
The lemma then follows with $K = 2^\eps + 1,$ after combining \eqref{a}, \eqref{b} and \eqref{c}.
\end{proof}

\begin{lemma}
\label{grad}
There exist a constant $C_\eps > 0 $, such that $$\left| \left( \sqrt{\phi_\eps(x+\xi)}\right){'}\right|^2 \leq C_\eps \phi_\eps(x), \hspace{2mm} \mbox{for almost all} \ \ x \in \R,   \ |\xi| \leq 1.$$
\end{lemma}
\begin{proof}
First we consider the case when $|x| >  2$. Using the fact that $|x+ \xi|^{\eps-2} < 1$ (which is true since $|x+\xi| > 1$ and $\eps < 2$) we obtain 
\begin{eqnarray}\label{phi_1}
\left| \left( \sqrt{\phi_\eps(x+\xi)}\right){'}\right|^2 = C_\eps \frac{|x+\xi |^{2\eps - 2}}{(1 + |x +\xi |^\eps)^3} \leq  C_\eps \frac{|x+\xi |^{\eps }}{(1 + |x +\xi |^\eps)^3}\leq  C_\eps\frac{1}{(1 + |x +\xi |^\eps)^2}  \nonumber\\  \leq  \frac{ C_\eps}{1 + |x +\xi |^\eps}.
\end{eqnarray}
Again since $|x| > 2$ and $|\xi| \leq 1$, this implies 
\begin{equation}\label{phi_2}
1 +|x+\xi |^\eps \geq 1+  \left(  |x| -1\right)^\eps \geq 1 + 
\left( \frac{|x|}{2}\right)^\eps  \geq \frac{1 + |x|^\eps}{2^\eps}.
 \end{equation}
 Combining the two equations \eqref{phi_1} and \eqref{phi_2}, we get the required inequality for $|x|>2$. \\
 
 Now let us consider the case $|x| \leq 2$. First note that if $|x+\xi | \leq 1$, then the required  inequality is a triviality as $\phi_{\eps}'(z) = 0$ on the set $ |z| \leq 1$.  So we consider the case when  $|x + \xi | >1$. Clearly in this case as $|x + \xi | \leq 3$, we have
 $$\left| \left( \sqrt{\phi_\eps(x+\xi)}\right){'}\right|^2 = C_\eps \frac{|x+\xi |^{2\eps - 2}}{(1 + |x +\xi |^\eps)^3}  \leq C_\eps |x + \xi |^{2\eps -2} \leq C_\eps\left\{
 \begin{array}{ll}
  3^{2\eps -2} \hspace{8.7mm}&\textrm{if} \   \eps \geq 1 \\
  1  \hspace{8.7mm} &\textrm{if}  \   0 < \eps  < 1. 
\end{array}
\right.$$
Lastly, as $|x|\leq 2$, it implies $\phi_{\varepsilon} \geq \frac{1}{2^{\varepsilon}+ 1}$. Therefore, for $|x| \leq 2$ and for some constant $C_{\eps} >0$, we get  
\begin{align*}
\left| \left( \sqrt{\phi_\eps(x+\xi)}\right){'}\right|^2  \leq C_{\varepsilon} \phi_{\varepsilon}(x).
\end{align*}
This  finishes the proof of the lemma.
\end{proof}\\

Now we are ready to prove the Theorem \ref{S}.\smallskip

\vspace*{0.5cm}
\noi\textbf{Proof of Theorem \ref{S}.} \\  First let us consider $S_s(\rho_{\eps,\lambda})(x)$. By Integrating in the variables  $y_2,  y_3,  .  \  ,  . \  ,   y_n $, we have 
\begin{align*}
S_s(\rho_{\eps,\lambda})(X) & = \int_{\R^n \setminus \{0\}}\left(\sqrt{\rho_{\eps,\lambda}(X+Y)} - \sqrt{\rho_{\eps,\lambda}(X)}\right)^2\frac{dY}{|Y|^{n+2s}} \nonumber\\ 
& = \int_{\R^n \setminus \{0\}}\left(\sqrt{\phi_\eps\left(\frac{x_1+y_1}{\lambda}\right)} - \sqrt{\phi_\eps\left(\frac{x_1}{\lambda}\right)}\right)^2\frac{dY}{|Y|^{n+2s}}  \nonumber\\ 
& = \Theta \int_{\R \setminus \{0\}}\left(\sqrt{\phi_\eps\left(\frac{x_1+y_1}{\lambda}\right)} - \sqrt{\phi_\eps\left(\frac{x_1}{\lambda}\right)}\right)^2\frac{dy_1}{|y_1|^{1+2s}},
\end{align*}
where $\Theta = \prod_{i=2}^n\Theta_i$ is as in Lemma \ref{connection}. Using the change of variable  $\lambda\tau = y_1 $, we get
$$S_s(\rho_{\eps,\lambda})(X) = \frac{\Theta}{\lambda^{2s}} \int_{\R \setminus \{0\}}\left(\sqrt{\phi_\eps\left(\frac{x_1}{\lambda} + \tau\right)} - \sqrt{\phi_\eps\left(\frac{x_1}{\lambda}\right)}\right)^2\frac{d\tau}{|\tau|^{1+2s}} := \frac{\Theta}{\lambda^{2s}} I_s\left(\frac{x_1}{\lambda}\right).$$
Clearly the theorem will be proved if we show that there exist a constant $C_\eps > 0,$
\begin{equation*}
\label{i}
I_s\left(x\right) \leq C_\eps \phi_\eps\left(x\right), \hspace{4mm} \forall x \in \R.
\end{equation*}
Since $\phi_\eps$ is even, using change of variable  it is  easy to show $I_s(x) = I_s(-x)$ and hence it is enough to consider the case when $x > 0$. We write $I_s(x) = I_1^s(x) + I_2^s(x)$, where 
$$I_1^s(x) =   \int_{ (-1,1) \setminus \{ 0 \}}\left(\sqrt{\phi_\eps\left(x + \tau\right)} - \sqrt{\phi_\eps\left(x\right)}\right)^2\frac{d\tau}{|\tau|^{1+2s}}$$
and 
$$I_2^s(x) =   \int_{ (-1,1)^c }\left(\sqrt{\phi_\eps\left(x + \tau\right)} - \sqrt{\phi_\eps\left(x\right)}\right)^2\frac{d\tau}{|\tau|^{1+2s}}.$$
Let us first estimate the term $I_2^s$ using Lemma \ref{sup}.
\begin{eqnarray}
\label{qq}
I_2^s(x) =   \int_{ (-1,1)^c }\left(\sqrt{\phi_\eps\left(x + \tau\right)} - \sqrt{\phi_\eps\left(x\right)}\right)^2\frac{d\tau}{|\tau|^{1+2s}} \
\leq 2 \int_{ (-1,1)^c }\left(\phi_\eps\left(x + \tau\right) + \phi_\eps\left(x\right)\right)\frac{d\tau}{|\tau|^{1+2s}}\nonumber\\
\leq  C \phi_\eps(x) + 2\int_{ (-1,1)^c }\phi_\eps\left(x + \tau\right) \frac{d\tau}{|\tau|^{1+2s}}  \leq  C \phi_\eps(x)  + C_\eps \phi_\eps(x) \int_{1}^\infty \frac{d\tau}{|\tau |^{1+2s-\eps}}. \nonumber
\end{eqnarray}
Since $\eps < 2s$, we get  
\begin{equation*}\label{i2}
I_2^s(x) \leq C_{\eps} \phi_\eps(x), \hspace{4mm} \forall x \in \R.
\end{equation*}
Now we estimate the term $I_1^s$. Using mean value theorem, there exists some $|\xi | \leq 1$ and  $x-\tau \leq x + \xi \leq x+\tau$, such that 
\begin{equation*}\label{kk}
\left|\sqrt{\phi_\eps(x+\tau)} - \sqrt{\phi_\eps(x)} \right| \leq \left| \left( \sqrt{\phi_\eps(x+\xi)}\right){'}\right|  |\tau |.
\end{equation*}
Plugging this expression in $I_1^s(x)$ and using Lemma \ref{grad} we finally get 
\begin{align*}
I_1^s(x) \leq C_{\varepsilon} \phi_{\varepsilon}(x).
\end{align*}
This completes the proof of the theorem.    \quad \hfill $\square$  \\

Now we present the proof of Theorem \ref{fd}.

\vspace*{0.5cm}

\textbf{Proof of Theorem \ref{fd}. }

By the hypothesis of the theorem, $u_{\ell}$ is the weak solution of the problem \eqref{Main Problem} where $$||f_{\ell}||_{L^2(\Omega_{\ell}\setminus \Omega_{\ell -1})} \leq K$$ for each $\ell$.
Taking $\lambda$ large enough we see by Theorem \ref{S} that the hypothesis of Proposition \ref{KAREN1} holds true with the function $\rho_{\eps, \lambda}$. Hence we obtain 
\begin{align}\label{thm Karen 1}
\int_{\Omega_{\ell}} u_{\ell}^2(X) \rho_{\eps, \lambda}(X) \ dX \leq C^1_{\gamma} \int_{\rn} f_{\ell}^2(X) \rho_{\eps, \lambda}(X) \ dX,
\end{align}
where, $C^1_{\gamma}$ is independent of $\ell$.
Now using assumption \eqref{force term condn}, inequality \eqref{thm Karen 1} and the definition of $\rho_{\eps, \lambda}$ we get
\begin{align*}
 \int_{\Omega_1} u_\ell^2 \ dX\leq C_1 \int_{\Omega_\ell} \rho_{\eps,\lambda} u_\ell^2 \ dX  \leq C_2 \int_{\Omega_\ell \setminus \Omega_{\ell-1}} \rho_{\eps,\lambda}  f_\ell^2 \ dX \leq \frac{C}{(\ell - 1)^{\eps}}, 
\end{align*}
for every $\varepsilon < 2s$. This completes the proof of the theorem.  \quad \hfill $\square$

\appendix

\section{Appendix}

As mentioned in the introduction, here we give a proof of Theorem \ref{main th} using Proposition \ref{KAREN1} with Theorem \ref{S} and we need a better regularity on the solution $u_{\infty}$ of the problem \eqref{Limiting  Problem} in order to prove this connection. The boundedness and regularity of $u_{\infty}$ follow from the next theorem obtained by Oton and Serra [see, \cite{oton}]. 

\begin{theorem} \label{regularity}
Let $\Omega$ be a bounded Lipschitz domain satisfying the exterior ball condition and $f \in L^{2}(\Omega)$. Let $u$ be the weak solution of type \eqref{Main Problem} where $f_{\ell}$ is replaced by $f$ and $\Omega_{\ell}$ is replaced by $\Omega$. 
\begin{enumerate}
\item[(i)] If $f \in L^{\infty}(\Omega)$ then $u \in C^{s}(\rn)$ and there exists $C_1>0$ such that 
 $$||u||_{C^{s}(\rn)} \leq C_1 ||f||_{L^{\infty}(\Omega)}.$$
 \item[(ii)] In addition, if  there exists $\beta > 0$ such that $\beta + 2s >2$ and $f \in C^{0,\beta}(\bar{\Omega})$, then 
 $u \in C^{2,2s+\beta-2}(\Omega)$ and there exists a constant $C_2>0$ such that 
 $$||u||_{C^{2, 2s+\beta-2}(\Omega)} \leq C_2 \left(||f||_{C^{0,\beta}(\bar{\Omega})} + ||u||_{C^{s}(\rn)}\right).$$ 
\end{enumerate}
\end{theorem}

Once we consider the above assumption on the force function $f$ to the problem \eqref{Limiting  Problem}, then the weak solution $u_{\infty}$  of \eqref{Limiting  Problem} is a classical solution and $u_{\infty} \in L^{\infty}(\R^{n-1})$. Next we state and prove the main theorem in this section. 

\begin{theorem} \label{connection th}
Let $s \in \left(\frac{1}{2},1\right)$, $f \in C^{0,\beta}(\bar{\Omega})$ for some $\beta>0$ such that $\beta + 2s >2$. Let $u_{\ell}$ be the solution of \eqref{Actual  Problem} and $u_{\infty}$ be the solution of \eqref{Limiting  Problem}. Then we have 
\begin{align*}
||u_{\ell} - u_{\infty}||_{L^2(\Omega_1)} \longrightarrow 0 \quad \mbox{as} \quad \ell \rightarrow \infty.
\end{align*}
\end{theorem} 

\begin{proof}
We need a very important auxiliary function in order to prove the theorem. Let $\psi_\ell: \R \rightarrow \R$ be a smooth bounded function on $\R$ such that 
\begin{equation*}\label{Main Problem 5}
 \psi_\ell = \left\{
 \begin{array}{ll}
  1 \hspace{8.7mm}&\textrm{in} \  (-\ell,  \ell)^c \\
  0    \hspace{8.7mm} &\textrm{on}  \  (-\ell +1, \ell-1). 
\end{array}
\right.
\end{equation*} 
We define 
\begin{align} \label{transformation}
w_\ell : = u_\ell - u_\infty + \psi_\ell u_\infty.
\end{align}

By the assumption on $f$ and $\psi_{\ell}$, using Theorem \ref{regularity} we observe that $(-\lap)^s u_{\infty} \psi_{\ell}$ can be evaluated point-wise in $\Omega_{\ell}$. First, we will show that $(-\lap)^s u_{\infty} \psi_{\ell} \in L^2(\Omega_{\ell})$. By the definition of fractional Laplacian we see
\begin{align*}
& (-\lap)^s u_{\infty}(X_2) \psi_{\ell}(x_1)\\ 
= & \frac{C_{n,s}}{2}\int_{\rn} \frac{2u_{\infty}(X_2)\psi_{\ell}(x_1)-u_{\infty}(X_2+Y_2)\psi_{\ell}(x_1+y_1) - u_{\infty}(X_2-Y_2)\psi_{\ell}(x_1-y_1) }{|Y|^{n+2s}} \ dY \\
= & \frac{C_{n,s}}{2}\psi_{\ell}(x_1)\int_{\rn} \frac{2u_{\infty}(X_2) - u_{\infty}(X_2+Y_2) - u_{\infty}(X_2-Y_2)}{|Y|^{n+2s}} \ dY \\
 & \quad + \frac{C_{n,s}}{2} \left[ \int_{B(0,1)} \frac{ u_{\infty}(X_2+Y_2)\{ \psi_{\ell}(x_1) - \psi_{\ell}(x_1+y_1) + y_1 \psi_{\ell}'(x_1)\}}{|Y|^{n+2s}} \ dY \right. \\
 & \hspace*{5cm} +\left. \int_{B(0,1)^c} \frac{u_{\infty}(X_2+Y_2) \{ \psi_{\ell}(x_1) - \psi_{\ell}(x_1+y_1) \}}{|Y|^{n+2s}} \ dY\right]\\
 \end{align*}
\begin{align*} 
 &\quad  + \frac{C_{n,s}}{2} \left[ \int_{B(0,1)} \frac{u_{\infty}(X_2-Y_2)\{\psi_{\ell}(x_1) - \psi_{\ell}(x_1-y_1) - y_1 \psi_{ \ell}'(x_1) \}}{|Y|^{n+2s}} \ dY  \right. \\
 & \hspace*{5cm}+\left.  \int_{B(0,1)^c} \frac{u_{\infty}(X_2-Y_2)\{ \psi_{\ell}(x_1) - \psi_{\ell}(x_1-y_1) \} }{|Y|^{n+2s}} \ dY\right] \\
=& \quad I_1 + I_2 + I_3,
\end{align*}
where, $I_1, I_2, I_3$ are the first, second and third integral respectively in the above expression. In the previous calculation we have used the fact that $y_1 \psi'_{\ell}(x_1) $ is an odd function of $y_1$ for each $Y_2$.  Now for $y_1 \in (-1,1)$ using the estimate $$ |\psi_{\ell}(x_1) - \psi_{\ell}(x_1\pm y_1) \pm \psi_{\ell}'(x_1)y_1| \leq y_1^2 ||\psi_{\ell}''||_{L^{\infty}(\R)}  $$ and the boundedness of $\psi_{\ell}$, we get 
\begin{align} \label{appendix1}
|I_2| + |I_3| \leq K ||u_{\infty}||_{L^{\infty}(\R^{n-1})} \left(  ||\psi_{\ell}||_{L^{\infty}(\R)} + ||\psi_{\ell}''||_{L^{\infty}(\R)} \right).
\end{align}
Next, we consider $I_1$ and as we did earlier in our article, integrating with respect to $y_1$ variable, we get $$|I_1| \leq C \ |\psi_{\ell}(x_1)| |(-\lap')^s u_{\infty}(X_2)| = C|\psi_{\ell}(x_1)| |f(X_2)|,$$ hence, combining this with \eqref{appendix1} we finally get for each $(x_1,X_2)\in \Omega_{\ell}$ 
\begin{align}\label{appendix2}
& |(-\lap)^s u_{\infty}(X_2) \psi_{\ell}(x_1)| \notag\\
 \leq & C ||\psi_{\ell}||_{L^{\infty}(\R)} \ |f(X_2)| + K ||u_{\infty}||_{L^{\infty}(\R^{n-1})} \left(  ||\psi_{\ell}||_{L^{\infty}(\R)} + ||\psi_{\ell}''||_{L^{\infty}(\R)} \right).
\end{align}
The inequality \eqref{appendix2} assures us  $(-\lap)^s u_{\infty} \psi_{\ell} \in L^2(\Omega_{\ell})$. Hence, from the definition of $u_{\ell}$ and $u_{\infty}$ we see that $v_{\ell}:= u_{\ell} - u_{\infty}$ satisfies the following equation : 
\begin{equation}\label{Main Problem 10}
  \left\{
 \begin{array}{ll}
 
\left(-\lap \right)^s v_\ell (X)= \left(-\lap\right )^s (\psi_\ell(x_1) u_\infty(X_2)) := \Psi_\ell(X)\hspace{8.7mm}&\textrm{in} \ \Omega_\ell , \\
  v_\ell = 0   \hspace{16.9mm} &\textrm{on}  \ \Omega_\ell^c,
\end{array}
\right.
\end{equation}
where, $ \Psi_\ell \in L^2(\Omega_{\ell})$.\smallskip

Now applying Proposition \ref{KAREN1} to the equation \eqref{Main Problem 10} with the family of functions $ \rho_{\eps,\lambda}$ for sufficiently large $\lambda$,   we get for any $\eps <2s$,
\begin{multline}\label{kkl}
\frac{1}{2}\int_{\Omega_1}(u_\ell - u_\infty)^2 = \int_{\Omega_1} \rho_{\eps,\lambda} v_\ell^2 \leq \int_{\Omega_\ell} \rho_{\eps,\lambda}v_\ell^2 \leq \int_{\Omega_\ell} \rho_{\eps,\lambda}\Psi_\ell^2 \\
 \leq \int_{\Omega_{\frac{\ell}{2}}} \rho_{\eps,\lambda} \Psi_\ell^2 + \int_{\Omega_\ell\setminus\Omega_{\frac{\ell}{2}}} \rho_{\eps,\lambda}\Psi_\ell^2  \leq 
 \int_{\Omega_{\frac{\ell}{2}}} \rho_{\eps,\lambda}\Psi_\ell^2 + \frac{\lambda^\eps}{\lambda^\eps + (\frac{\ell}{2})^\eps}\int_{\Omega_\ell\setminus\Omega_{\frac{\ell}{2}}}\Psi_\ell^2 . 
\end{multline}
First we estimate the function $\Psi_\ell$ on the set $\Omega_{\frac{\ell}{2}}$ point-wise. By definition
\begin{align*}
\Psi_\ell(X) & := \left( -\lap \right)^s(u_\infty(X_2)\psi_\ell(x_1)) \\ 
& = \frac{1}{2}\int_{\R^n} \frac{2 u_\infty(X_2)\psi_\ell(x_1)- u_\infty(X_2+ Y_2)\psi_\ell(x_1 +y_1) - u_\infty(X_2- Y_2)\psi_\ell(x_1 -y_1) }{|Y|^{n+2s}} dY\\
&= \frac{1}{2}\int_{\R^n} \frac{- u_\infty(X_2+ Y_2)\psi_\ell(x_1 +y_1) - u_\infty(X_2- Y_2)\psi_\ell(x_1 -y_1) }{|Y|^{n+2s}} dY .
\end{align*}
Clearly as $\omega$ is bounded we can find a ball $B_R$ of radius $R$ in $\R^{n-1}$ such that $\omega \subset B_{R}$. Hence, as $X \in \Omega_{\frac{\ell}{2}}$, we see that $$u_{\infty}(X_2\pm Y_2)\psi(x_1\pm y_1)=0 \quad \mbox{where} \quad Y \in \left(\left(-\frac{\ell}{2} +1, \frac{\ell}{2} -1\right)^c \times B_{2R} \right)^c.$$ 
Using this fact we get 
\begin{align}\label{kkl2}
&|\Psi_\ell(X) |  \leq 2 ||u_\infty ||_{L^{\infty}(\R^{n-1})}||\psi_{\ell}||_{L^{\infty}(\R)} \int_{\left(-\frac{\ell}{2} +1, \frac{\ell}{2} -1\right)^c \times B_{2R}}\frac{dY}{|Y|^{n+2s}} \notag\\
&\leq 2 ||u_\infty ||_{L^{\infty}(\R^{n-1})}||\psi_{\ell}||_{L^{\infty}(\R)} |B_{2R}| \int_{\left(-\frac{\ell}{2} +1, \frac{\ell}{2} -1\right)^c}\frac{dy_1}{|y_1|^{n+2s}} \leq
\frac{C_1 ||u_\infty ||_{L^{\infty}(\R^{n-1})}||\psi_{\ell}||_{L^{\infty}(\R)}}{\ell^{n-1+2s} }.
\end{align}
Using the estimate \eqref{kkl2} we see there exists a constant $C_2>0$ depending on the $L^{\infty}$ norm of $u_{\infty}$ and $\psi_{\ell}$ such that 
\begin{align} \label{appendix3}
\int_{\Omega_{\frac{\ell}{2}}} \Psi_{\ell}^2(X) \ dX \leq  \frac{C_2}{\ell^{2n-3+4s}}.
\end{align}
Finally, using the inequality \eqref{appendix3} and the fact $||\Psi_{\ell}||^2_{L^2(\Omega_{\ell})} \leq C \ell $ to the expression \eqref{kkl}, we obtain  
 $$\int_{\Omega_1} (u_\ell(X)-u_\infty(X))^2 dX \leq  \frac{C_2}{\ell^{n-2+2s} } +  \frac{C}{\ell^{\eps-1}}.$$
Since $\eps$ is any arbitrary number less than $2s$ and $s>\frac{1}{2}$, we take $\eps$ as $1<\eps <2s $. Now this implies 
that, as $\ell \rightarrow \infty$,
$$\int_{\Omega_1} (u_\ell-u_\infty)^2  \longrightarrow 0.$$
This completes the proof.
\end{proof}

\bigskip

\noindent\textbf{Acknowledgments.}  The research work of the second author is supported by "Innovation in Science Pursuit for Inspired Research (INSPIRE)" under the IVR Number: 20140000099.


\begin{thebibliography}{200}
\small{





\bibitem{kassman} M. Felsinger, M. Kassmann, and P. Voigt; The Dirichlet problem for nonlocal operators. \textit{ Math. Z.} 279 (2015), no. 3-4, 779--809.
\vspace{-.25cm}

\bibitem{b}
M. Chipot; $\ell$ goes to plus infinity, Birkh\"auser, 2002. 
\vspace{-.25cm}


\bibitem{CR}
M. Chipot  and  A. Rougirel; On the asymptotic behaviour of the solution of elliptic problems in cylindrical domains becoming unbounded. \textit{Commun. Contemp. Math.} 4 (2002), no. 1, 15-44. 
\vspace{-.25cm}

\bibitem{d} 
M. Chipot  and A. Rougirel;  Remarks on the asymptotic behaviour of the solution to parabolic problems in domains becoming unbounded, \textit{ Nonlinear Analysis }
47, p. 3--12, 2001.
\vspace{-.25cm}

\bibitem{pro} M. Chipot, A. Mojsic and P. Roy;  On some variational problem set on domains tending to infinity, \textit{to appear}.
\vspace{-.25cm}

\bibitem{pr} M. Chipot, P. Roy and  I. Shafrir;   Asymptotics of eigenstates of elliptic problems with mixed boundary data on domains tending to infinity,  \textit{Asymptotic Analysis,} 85, no. 3-4, 199--227, 2013.
 \vspace{-.25cm}
 
 \bibitem{CS1}
M. Chipot and S. Mardare;  Asymptotic behaviour of the Stokes problem in cylinders becoming unbounded in one direction. \textit{ J. Math. Pures Appl.} (9) 90 (2008), no. 2, 133-159.
 \vspace{-.25cm}
 
\bibitem{CS2}
M. Chipot  and S. Mardare;  The Neumann problem in cylinders becoming unbounded in one direction. \textit{J. Math. Pures Appl.} 2015.
 \vspace{-.25cm}
  
 \bibitem{ka}
 M. Chipot  and  K. Yeressian;  On the asymptotic behavior of variational inequalities set in cylinders. \textit{ Discrete Contin. Dyn. Syst.} 33 , no. 11--12, 2013.
 \vspace{-.25cm}
 
\bibitem{karen}  M. Chipot and  K. Yeressian;   Exponential rates of convergence by an iteration technique. \textit{ C. R. Math. Acad. Sci. }Paris 346, no. 1-2, 21--26, 2008.
\vspace{-.25cm}



\bibitem{NGV}
D. Nezza, E. Palatucci and G. Valdinoci Enrico; Hitchhiker's guide to the fractional Sobolev spaces. \textit{ Bull. Sci. Math. }136 (2012), no. 5, 521-573.
\vspace{-.25cm}

\bibitem{oton}
X. Ros-Oton and J. Serra;  The Dirichlet problem for the fractional Laplacian: regularity up to the boundary. \textit{J. Math. Pures Appl.} (9) 101 (2014), no. 3, 275-302.
\vspace{-.25cm}

\bibitem{Sen1}
S. Guesmia; Some results on the asymptotic behavior for hyperbolic problems in cylindrical domains becoming unbounded, \textit{J. Math. Anal. Appl.} 341 (2008), no. 2, 1190-2012.
\vspace{-.25cm}

\bibitem{Sen2}
S. Guesmia; Some convergence results on quasi-linear parabolic boundary value problems in cylindrical domains of large size, \textit{Nonlinear Anal.} 70 (2009), no. 9, 3320-3331.
\vspace{-.25cm}

\bibitem{Karen} K. Yeressian;  Asymptotic behavior of elliptic nonlocal equations set in cylinders, \textit{ Asymptot. Anal.} 89 (2014), no. 1-2, 21--35.
\vspace{-.25cm}





}
\end{thebibliography}
\end{document}